\documentclass[11pt]{article}

\usepackage[margin=1in,a4paper]{geometry}
\usepackage[utf8]{inputenc}
\usepackage{graphicx}
\usepackage{mathtools}  
\mathtoolsset{showonlyrefs}

\usepackage[margin=1in,a4paper]{geometry}
\usepackage{amssymb,amsmath}
\usepackage{amsthm} 

\usepackage{xcolor}
\usepackage{esint}
\usepackage[colorlinks,linkcolor=blue, urlcolor  = blue,citecolor=blue,pagebackref,hypertexnames=false, breaklinks]{hyperref}

\usepackage[nottoc,notlot,notlof]{tocbibind}
\def\Xint#1{\mathchoice
    {\XXint\displaystyle\textstyle{#1}}%
    {\XXint\textstyle\scriptstyle{#1}}%
    {\XXint\scriptstyle\scriptscriptstyle{#1}}%
    {\XXint\scriptscriptstyle\scriptscriptstyle{#1}}%
    \!\int}
    \def\XXint#1#2#3{{\setbox0=\hbox{$#1{#2#3}{\int}$}
    \vcenter{\hbox{$#2#3$}}\kern-.5\wd0}}
    \def\fint{\Xint-}

\theoremstyle{plain}
\newtheorem{thm}{Theorem}[section]
\newtheorem{lem}[thm]{Lemma}
\newtheorem{cor}[thm]{Corollary}
\newtheorem{prop}[thm]{Proposition}

\newtheorem{example}[thm]{Example}

\DeclareMathOperator{\supp}{supp}

\theoremstyle{definition}
\newtheorem{defn}[thm]{Definition}

\theoremstyle{remark}
\newtheorem{remark}[thm]{Remark}
\newcommand{\bremark}{\begin{remark} \em}
\newcommand{\eremark}{\end{remark} }
\usepackage{color}

\allowdisplaybreaks[4]

\hbadness=\maxdimen

\title{Sobolev spaces and Poincar\'e  inequalities on the Vicsek fractal}
\author{Fabrice Baudoin\footnote{Partly supported by the NSF grant DMS~1901315.} and Li Chen\footnote{Partly supported by Simons collaboration grant \#853249.}}
\date{}

\begin{document}

\maketitle

\begin{abstract}
In this paper we prove that several natural approaches to Sobolev spaces coincide on the Vicsek fractal. More precisely, we show that the metric approach of Korevaar-Schoen, the approach by limit of discrete $p$-energies and the approach by limit of Sobolev spaces on cable systems all yield the same functional space with equivalent norms for $p>1$. As a consequence we  prove that the Sobolev spaces form a real interpolation scale.   We also obtain $L^p$-Poincar\'e inequalities  for all values of $p \ge 1$.
\end{abstract}

\tableofcontents

\section{Introduction}

The theory of Sobolev spaces on abstract metric measure spaces has attracted a lot of attention in the last few decades and the upper gradient approach has proved to be one of the most successful approaches to develop a rich theory, see \cite{MR3363168} and the references therein. However, due to the generic lack of rectifiable curves between points, the approach  by upper gradients techniques is not relevant anymore in the context of fractals.

For  fractals, the theory of Sobolev spaces can  be developed from several viewpoints. A first natural approach is through the study of discrete p-energies as in Herman-Peirone-Strichartz \cite{HPS}, Hu-Ji-Wen \cite{HuJiWen}, and more recently Cao-Gu-Qiu \cite{CaoGuQiu},  Kigami \cite{MR4175733,Kigami}, and  Shimizu \cite{Shimizu}. This approach makes a  crucial use of the fact that fractals can be approximated by discrete spaces in a somehow canonical way. A second natural and purely metric approach is based on the Korevaar-Schoen approach and defines Sobolev spaces as endpoints in a scale of Besov-Lipschitz spaces, see \cite{MR4196573}, \cite{BV3} and  \cite{BC2020}. Finally, a third natural approach in the context of nested fractals, is to define Sobolev spaces as the  functional spaces whose traces on approximating cable systems are the Sobolev spaces of that cable system. 

A  first goal of this paper is to prove that those three approaches are actually equivalent and, for $p>1$,  all yield the same space $W^{1,p}$ in a popular example of nested fractal: the Vicsek set, see Figure \ref{figure1}. We achieve this goal in Theorem \ref{carac Sobolev} below. The case $p=1$ is discussed separately in the text, and the space $W^{1,1}$ we single out is a strict subspace of the space of BV functions that was defined in \cite{BV3}  in the general context of Dirichlet spaces with sub-Gaussian heat kernel estimates.  We note that some parts of the proof of  Theorem \ref{carac Sobolev} make  use of the notion of piecewise affine function which is specific to the Vicsek set setting. In particular, our arguments do not easily extend to the case of other nested fractals like the Sierpinski gasket and the study of this fractal is left to a later work.

We  also prove the following family of Poincar\'e inequalities on the Vicsek set $K$: for $p \ge  1$, there exist constants $c,C>0$ such that for any $f \in W^{1,p}$,  $x_0\in K$ and $r>0$ we have:
\[
\int_{B(x_0,r)} \left|f(x)-\frac1{\mu(B(x_0,r))}\int_{B(x_0,r)}f d\mu\right|^p d\mu(x) 
\le Cr^{p-1+d_h} \| f \|_{W^{1,p}(B(x_0,cr))}^p,
\]
where $d_h$ is the Hausdorff dimension of the Vicsek set. The exponent $p-1+d_h$ is sharp as follows from Remark \ref{sharp poinc}. Our proof is based on the introduction of a notion of weak gradient on the Vicsek set which is similar to the notion of exterior derivative considered in \cite{BK} for cable systems (or more generally spaces of Hino index one).   We note that the study of Poincar\'e inequalities on some nested fractals including the Vicsek set was undertaken in \cite{BC2020} where some stronger inequalities were proved in the  range $1 \le p \le 2$ using  heat semigroups techniques instead. The case $p>2$ was let open in \cite{BC2020} and therefore the present paper settles the question of the validity of the Poincar\'e inequality for all the range $p \ge 1$. 

The Poincar\'e inequalities we obtain imply that any Sobolev function $f \in W^{1,p}$, $p>1$ satisfies a Lusin-H\"older estimate:
\[
| f(x)-f(y)| \le  d(x,y)^{1-\frac1p+\frac{d_h}{p} }( g(x) +g(y))
\]
where  $g$ is a function in weak $L^p$. We show that the function $g$ can not be in $L^p$ however, unless the function $f$ is constant. This shows in particular that the Haj\l{}asz-Sobolev space introduced by Hu in \cite{MR1972194} is trivial at the critical exponent $\alpha_p$ for the case of the Vicsek set.  Therefore, in the context of fractals,  the approach to Sobolev spaces due to P. Haj\l{}asz \cite{Haj} does not  yield a satisfactory theory.

Another objective of the paper is to study the real interpolation properties of the Sobolev spaces and obtain, for the Vicsek set, an analogue of the main result of the paper by Gogatishvili-Koskela-Shanmugalingam \cite{GKS}.  Specifically, we prove that for every $p > 1$ and $0\le \alpha \le\alpha_p:=1-\frac1p+\frac{d_h}{p} $
\[
\mathcal{B}^{\alpha}_{p,\infty}=(L^p, W^{1,p})_{\alpha/\alpha_p,\infty},
\]
where  $\mathcal{B}^{\alpha}_{p,\infty}$ is a Besov-Lipschitz space which coincides with a heat semigroup based Besov space introduced and studied in \cite{BV1}, see also \cite{Gri} \cite{MR3420351}, \cite{MR2574998} for further characterizations and properties of the Besov-Lipschitz spaces. We also prove that the Sobolev spaces form, with respect to the parameter $p \ge 1$ a real interpolation scale, i.e., for $1 \le p_1 <p < p_2 \le +\infty$,
\[
W^{1,p}=(W^{1,p_1},W^{1,p_2})_{\theta,p}
\]
where $\theta \in (0,1)$ is such that
\[
\frac{1}{p}=\frac{1-\theta}{p_1}+\frac{\theta}{p_2}.
\]
Note that by the reiteration theorem we therefore obtain the full interpolation theory for the spaces $\mathcal{B}^{\alpha}_{p,\infty}$ including  the endpoints $\mathcal{B}^{\alpha_p}_{p,\infty}=W^{1,p}$.

\paragraph{Notations:}
\begin{enumerate}
    \item Throughout the paper, we use the letters $c,C, c_1, c_2, C_1, C_2$  to denote positive constants which may vary from line to line. 
    
    \item For two non-negative functionals $\Lambda_1, \Lambda_2$ defined on a functional space $\mathcal F$, the notation $\Lambda_1(f) \simeq \Lambda_2(f)$ indicates that there exist two constants $C_1, C_2>0$ such that for every $f\in \mathcal F$, $C_1\Lambda_1(f)\le \Lambda_2(f)\le C_2 \Lambda_1(f)$.
    
\item For any Borel set $A$ and any measurable function $f$, we write  the average of $f$ on the set $A$ as
\[
\fint_A f(x) d\mu(x):=\frac1{\mu(A)}\int_A f(x) d\mu(x).
\]
\end{enumerate}

\section{Preliminaries and notations}
\subsection{Vicsek set}

\begin{figure}[htb]\label{figure1}
 \noindent
 \makebox[\textwidth]{\includegraphics[height=0.22\textwidth]{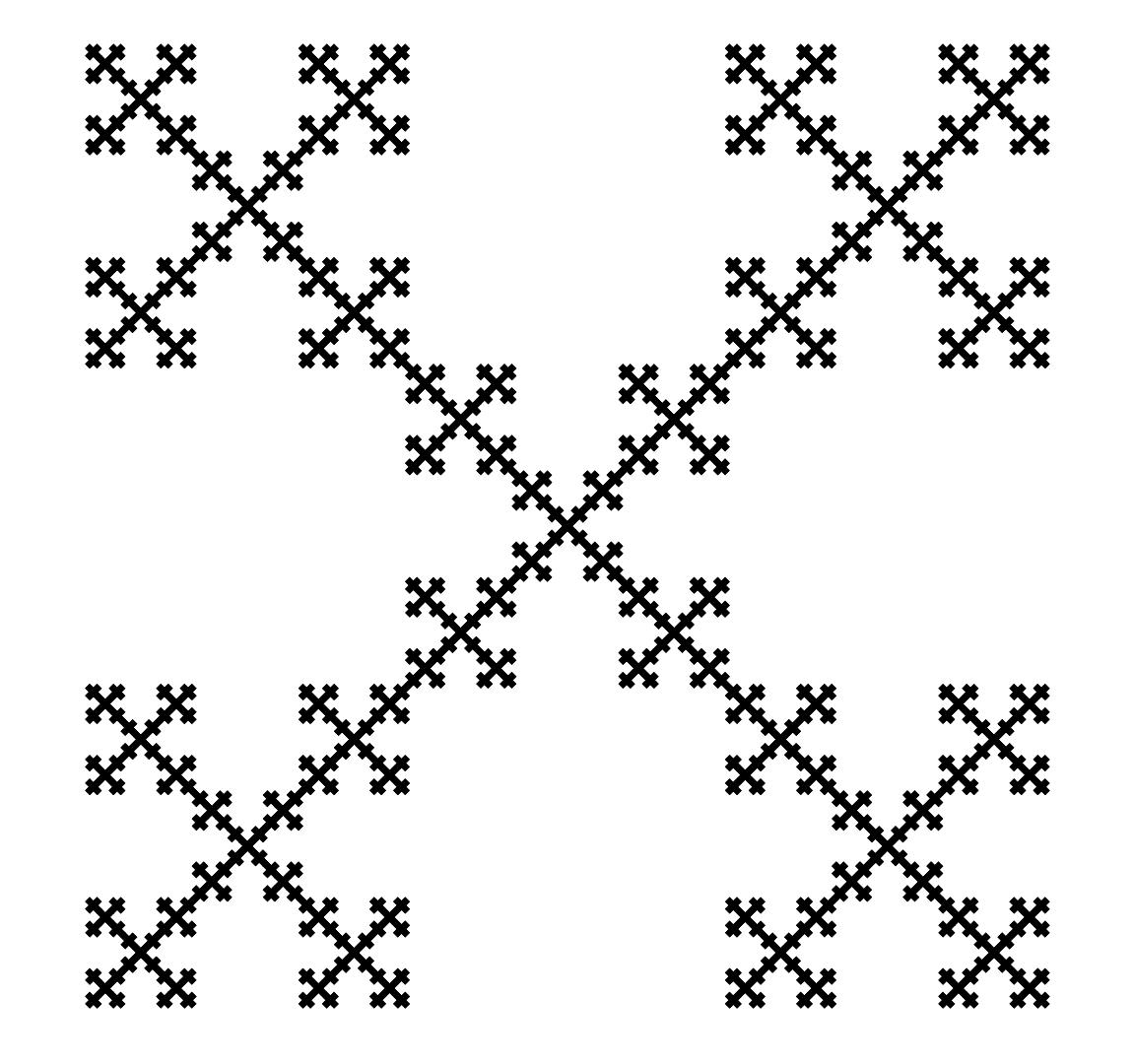}}
  	\caption{Vicsek set}
\end{figure}
Let $q_1=(-\sqrt{2}/2,\sqrt{2}/2)$, $q_2=(\sqrt{2}/2,\sqrt{2}/2)$ , $q_3=(\sqrt{2}/2,-\sqrt{2}/2)$, and $q_4=(-\sqrt{2}/2,-\sqrt{2}/2)$ be the 4 corners of the unit square and let $q_5=(0,0)$ be the center of that square. Define $\psi_i(z)=\frac13(z-q_i)+q_i$ for $1\le i\le 5$. Then the Vicsek set $K$ is the unique non-empty compact set such that 
\[
K=\bigcup_{i=1}^5 \psi_i(K).%=:\Psi(K).
\]

%\subsection{Vicsek graphs and skeleton}

\begin{figure}[htb]\label{figure2}
 \noindent
 \makebox[\textwidth]{\includegraphics[height=0.20\textwidth]{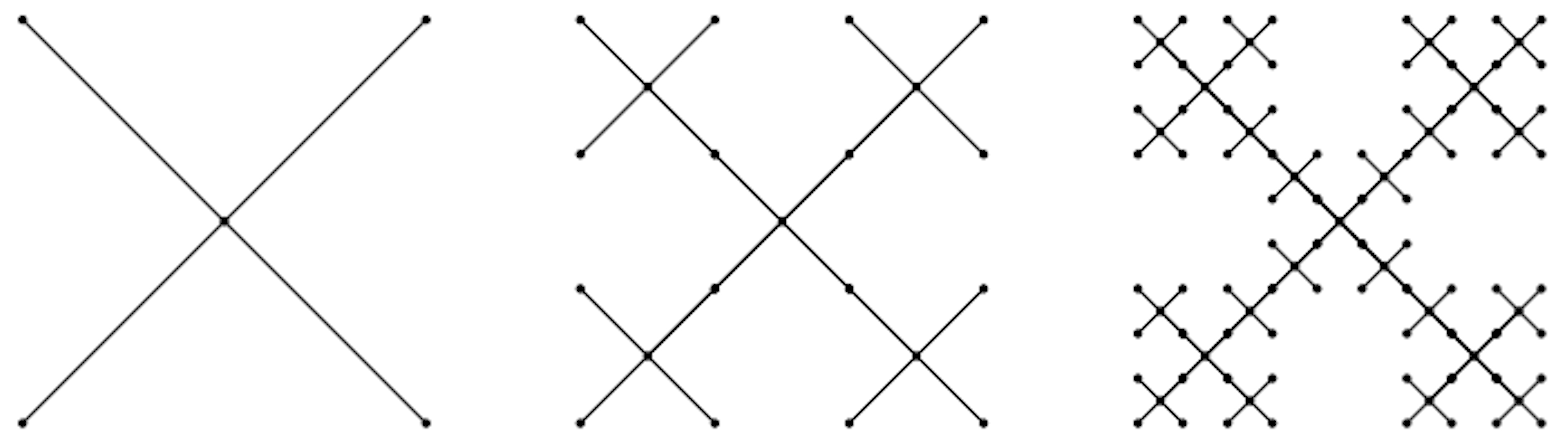}}
  	\caption{Vicsek cable systems $\bar{V}_0$, $\bar{V}_1$ and $\bar{V}_2$}
\end{figure}

Denote $W=\{1, 2,3,4, 5\}$ and $W_n=\{1, 2,3,4, 5\}^n$ for $n\ge 1$. For any $w=\{i_1, \cdots, i_n \}\in W_n$, we denote by $\Psi_w$ the contraction mapping $\psi_{i_1}\circ \cdots\circ \psi_{i_n}$ and write $K_w:=\Psi_w(K)$. The set $K_w$ is called  an $n$-simplex.

Let $V_0=\{q_1, q_2, q_3,q_4,q_5\}$. 
We define a sequence of  sets of vertices  $\{V_n\}_{n\ge 0}$ inductively by
\[
V_{n+1}=\bigcup_{i=1}^5 \psi_i(V_n).
\]
For any $w=\{i_1, \cdots, i_n \}\in W_n$, we will denote $V_n^w=\Psi_w(V_n)$.
Let then $\bar{V}_0$ be the cable system\footnote{Cable systems are also sometimes called quantum graphs or metric graphs in the literature} with vertices $V_0=\{q_1,q_2, q_3, q_4,q_5\}$
 and  consider the sequence of cable systems  $\bar{V}_n$ with vertices in $V_n$ inductively defined as follows.  The first cable system is  $\bar{V}_0$ and then 
\[
\bar{V}_{n+1} = \bigcup_{i=1}^5 \psi_i (\bar{V}_n).
\]
Note that $\bar{V}_n \subset K$ and that $K$ is the closure of $\cup_{n \ge 0} \bar{V}_n $. 
The set 
\[
\mathcal{S}= \bigcup_{n \ge 0} \bar{V}_n
\]
is called the skeleton of $K$ and is dense in $K$.   Therefore we have a natural increasing sequence of Vicsek cable systems $\{\bar{V}_n\}_{n\ge 0}$ whose edges have length $3^{-n}$ and whose set of vertices is $V_n$ (see Figure \ref{figure2}). From this viewpoint, the Vicsek set $K$ is seen as a limit of the cable systems $\{\bar{V}_n\}_{n\ge 0}$. 

 If $u,v$ are adjacent vertices in  $\bar{V}_n$ we will write $u \sim v$. We then denote by $\mathbf{e}(u,v)$  the edge in $\bar{V}_n$ connecting $u$ to $v$. We will say that $u \le v$ if the geodesic distance from the center $(0,0)$ of the Vicsek set  to $u$ in  $\bar{V}_n$ is less than the geodesic distance from $(0,0)$ to $v$.

\subsection{Geodesic distance and measures on the Vicsek set}

On $K$ we will consider the  geodesic distance $d$. For $x,y \in \bar{V}_n$, $d(x,y)$ is defined as the length of the geodesic path between $x$ and $y$ and $d$ is then extended by continuity to $K \times K$. The geodesic distance $d$ is then bi-Lipschitz equivalent to the restriction of the Euclidean distance to $K$.  The Hausdorff  measure $\mu$ is the normalized measure on $K$ such that $i_1, \cdots, i_n \in \{1,2,3, 4,5\}$
\[
\mu(\psi_{i_1} \circ \cdots \circ \psi_{i_n} (K))=5^{-n}.
\]
The Hausdorff dimension of $K$ is then $d_h=\frac{\log 5}{\log 3}$ and the metric space $(K,d)$ is $d_h$-Ahlfors regular in the sense that there exist constants $c,C>0$ such that for every $x \in K$, $r \in [0, \mathrm{diam}\,  K]$,
\[
c \,r^{d_h} \le \mu (B(x,r)) \le Cr^{d_h},
\]
where $B(x,r)=\left\{ y \in K, d(x,y) \le r \right\}$ denotes the closed ball with center $x$ and radius $r$ and $\mathrm{diam} \, K=2$ is the diameter of $K$.
%
%Consider on $K$ the counting probability measure 
%\[
%\mu_n:=\frac1{\# V_n}\sum_{x\in V_n}\delta_x,
%\]
%where $\delta_x$ denotes the Dirac measure with support $\{x\}$ (i.e. for a set  $A$, $\delta_x(A)=1$ or $0$ depending on whether or not $x \in A$) . Without abuse of notation, we equip $V_n$ with the measure $\mu_n$ in the sense that  for any $f\in \ell(V_n)$, 
%\[
%\sum_{x\in V_n}f(x)\mu_n(x)=\frac{1}{\#V_n}\sum_{x\in V_n}f(x),
%\]
%where $\ell(V_n)$ denotes the set of functions $f:V_n \to \mathbb R$. The following lemma is elementary (and proved as in the Sierpi\'nski gasket case, see \cite[Lemma 1.1]{BarlowPerkins}).
%\begin{lem}
%%[\protect{\cite[Lemma 1.1]{BarlowPerkins}}]
%\label{lem:measure}
%The sequence of measures $\{\mu_n\}_{n\ge 0}$ converges to the normalized Hausdorff measure $\mu$ on the Vicsek set $K$ in the weak topology. That is, 
%\[
%\lim_{n\to \infty} \sum_{x\in V_n}f(x)\mu_n(x)=\lim_{n\to \infty}\int_K fd\mu_n= \int_K fd\mu, \quad \forall f\in C(K),
%\]
%where $C(K)$ denotes the set of continuous functions on $K$.
%\end{lem}

There is also a reference measure $\nu$ on the skeleton $\mathcal{S}$, the Lebesgue measure. It is characterized by the property that for every edge $\mathbf{e}$  $\bar{V}_n$ connecting two neighboring vertices:
\[
\nu(\mathbf{e})= 3^{-n}.
\]
The measure $\nu$ is not finite (because the skeleton has infinite length) but it is $\sigma$-finite on the $\sigma$-field generated by the $\mathbf{e}(x,y)$, $x,y \in V_n$, $x \sim y$, $n \ge 0$. The measure  $\nu$ is not a Radon measure neither since the measure of any  ball with positive radius is infinite.
From its definition, it is also clear that $\nu$ is singular with  respect to the Hausdorff measure $\mu$ since  the  skeleton has $\mu$-measure zero. For further comments about this measure $\nu$, we also refer for instance to the introduction of \cite{CSW}.

\subsection{Korevaar-Schoen-Sobolev spaces on the Vicsek set}

We now introduce the definitions of the Korevaar-Schoen-Sobolev spaces on the Vicsek set, following the previous works \cite{BV3,BC2020}. In particular, in this paper, we will use the following notations and definitions.

\begin{defn}\label{def:p-KSS}
Let $p > 1$. The Korevaar-Schoen-Sobolev space $W^{1,p}(K)$ is defined by
%by using the $L^p$ Korevaar-Schoen energy:
\[
W^{1,p}(K)=\left\{ f \in L^p(K,\mu), \,  \limsup_{r\to 0^+} \frac{1}{r^{\alpha_p}}\Bigg( \int_K\int_{B(x,r)}\frac{|f(y)-f(x)|^p}{ \mu(B(x,r))}\, d\mu(y)\, d\mu(x)\Bigg)^{\frac1p} <+\infty \right\},
\]
where $\alpha_p=1-\frac1p+\frac{d_h}{p}$. The semi-norm of $f\in W^{1,p}(K)$ is given by 
\[
\|f\|_{W^{1,p}(K)}:=\limsup_{r\to 0^+} \frac{1}{r^{\alpha_p}}\Bigg( \int_K\int_{B(x,r)}\frac{|f(y)-f(x)|^p}{ \mu(B(x,r))}\, d\mu(y)\, d\mu(x)\Bigg)^{\frac1p},
\]
and for a Borel subset $A \subset  K$ we will denote
\[
\|f\|^p_{W^{1,p}(A)}:=\limsup_{r\to 0^+} \frac{1}{r^{p\alpha_p}} \int_A\int_{B(x,r) \cap A}\frac{|f(y)-f(x)|^p}{ \mu(B(x,r))}\, d\mu(y)\, d\mu(x).
\]
\end{defn}

\begin{remark}
It is easy to prove, see \cite{BV3}, that for a function $f \in L^p(K,\mu)$, if $ \| f \|^p_{W^{1,p}(K)}<+\infty$ then
\[
 \sup_{r >0} \frac{1}{r^{p\alpha_p}} \int_K\int_{B(x,r) }\frac{|f(y)-f(x)|^p}{ \mu(B(x,r))}\, d\mu(y)\, d\mu(x) <+\infty.
\]
We will  further prove in Corollary \ref{cor:sup-limsup} that for every $p>1$, there exists a constant $C>0$ such that for every $f \in W^{1,p}(K)
$
\begin{equation*}
  \sup_{r >0} \frac{1}{r^{p\alpha_p}} \int_K\int_{B(x,r) }\frac{|f(y)-f(x)|^p}{ \mu(B(x,r))}\, d\mu(y)\, d\mu(x) 
 \le  C \| f \|^p_{W^{1,p}(K)}.
\end{equation*}
As a consequence,  using the $\sup$ or $\limsup$ in the definition of $W^{1,p}(K)$ eventually yields the same space with equivalent semi-norms.
\end{remark}

\begin{remark}\label{rem:continuity}
It follows from \cite{MR4196573} that for every $p>1$, $W^{1,p}(K) \subset C(K)$ where $C(K)$ denotes the set of continuous functions on $K$. More precisely, any function in $W^{1,p}(K)$ has a continuous $L^p$ representative, so in the sequel we will look at $W^{1,p}(K)$ as a subspace of $C(K)$ when $p>1$.
\end{remark}

\begin{remark}
The space $W^{1,2}(K)$ is the domain of the canonical self-similar Dirichlet form on $K$, see \cite{MR3420351}.
\end{remark}

\begin{defn}
For $p=\infty$, we define $W^{1,\infty}(K)$ to be the set of Lipschitz continuous functions on $K$ equipped with the seminorm
\[
\| f \|_{W^{1,\infty}(K)}= \sup_{x,y \in K, x\neq y} \frac{|f(x)-f(y)|}{d(x,y)}.
\]
\end{defn}

When $p=1$, the Korevaar-Schoen approach yields the space $BV(K)$ of bounded variations function on $K$, see Definition \ref{def:BV} and  \cite{BV3}. The definition of $W^{1,1}(K)$ is given below (see Definition \ref{defW11}) and is motivated from Theorem \ref{carac Sobolev} and Section \ref{inter sobo}.

\subsection{Discrete p-energies}

Another natural approach to Sobolev spaces on  fractals is by using limits of discrete $p$-energies, see \cite{HPS} and the more recent \cite{CaoGuQiu}, \cite{BC2020}, \cite{Kigami}, \cite{Shimizu}. For $1 \le p <+\infty$, the discrete $p$-energy of  a function  $f \in C(K)$ is defined as
\[
\mathcal{E}_p^m(f):=\frac{1}{2} 3^{(p-1)m}  \sum_{x,y \in V_m, x \sim y} |f(x)-f(y)|^p.
\]
For $p=+\infty$, we define
\[
\mathcal{E}_\infty^m(f):= 3^{m} \max_{x,y \in V_m, x \sim y}   |f(x)-f(y)|.
\]
Here the constant $3$ is in fact the resistance scale factor of the Vicsek set $K$.
%and $x\sim_{m}y$ denotes that $x$ and $y$ are neighbors in $V_m$. 
For a subset $A\subset K$ we define for $1 \le p <+\infty$
\[
\mathcal{E}_{A,p}^m(f):=\frac{1}{2} 3^{(p-1)m} \sum_{x,y\in A \cap V_m, x \sim y} |f(x)-f(y)|^p, 
\]
and
\[
\mathcal{E}_{A,\infty}^m(f):= 3^{m} \max_{x,y \in A \cap V_m , x \sim y}   |f(x)-f(y)|.
\]

The subset $A \subset K$ will be called convex if for any two points $x,y \in  A\cap \mathcal S$ the geodesic path connecting $x$ to $y$ is included in $A\cap \mathcal S$. For instance, any ball $B(x_0,r)$ in $K$ is convex. If $A$ is convex, as a consequence of the basic  inequalities
\[
|x+y+z|^p\le 3^{p-1}(|x|^p+|y|^p+|z|^p), \quad |x+y+z| \le 3 \max ( |x|,|y|,|z|)
\]
and of  the tree structure of $A \cap V_m$ we always have for $1 \le p \le +\infty$
\begin{equation}\label{eq:increaseEn}
\mathcal{E}_{A,p}^m(f)\le \mathcal{E}_{A,p}^n(f), \quad \forall \, m,n\in \mathbb N,m\le n.
\end{equation}
Moreover, from this fact we deduce that   
\begin{equation}\label{eq:lim}
\lim_{n\to \infty}\mathcal{E}_{A,p}^n(f)=\sup_{n\ge 0} \mathcal{E}_{A,p}^n(f)=\limsup_{n\to \infty}\mathcal{E}_{A,p}^n(f)=\liminf_{n\to \infty}\mathcal{E}_{A,p}^n(f),
\end{equation}
where the above  quantities are in $\mathbb{R}_{\ge 0} \cup \{+\infty \}$.

\begin{defn}\label{def:p-energyVS}
Let $1 \le p \le +\infty$. For any convex subset $A\subset K$ and $f \in C(K)$, we define the (possibly infinite) $p$-energy on $A$ by
\[
\mathcal E_{A,p}(f):=\lim_{m\to \infty} \mathcal E_{A,p}^m(f).
\]
\end{defn}
\begin{defn}
Let $1 \le p \le +\infty$. We define
\[
\mathcal{F}_p=\left\{ f \in C(K), \sup_{m \ge 0} \mathcal{E}_p^m(f) <+\infty \right\}
\]
and consider on $\mathcal F_p$ the seminorm
\[
\| f \|_{\mathcal{F}_p}=
\begin{cases}
\sup_{m \ge 0} \mathcal{E}_p^m(f)^{1/p}, \, 1 \le p <+\infty \\
\sup_{m \ge 0}  \mathcal{E}_\infty^m(f).
\end{cases}
\]
\end{defn}

\subsection{Piecewise affine functions}

A continuous function $\Phi:K \to \mathbb{R}$ is called $n$-piecewise affine, if there exists $n \ge 0$ such that $\Phi$ is piecewise affine on the cable system $\bar{V}_n$ (i.e linear between the vertices of $\bar{V}_n$)  and constant on any connected component of $\bar{V}_m \setminus \bar{V}_n$ for every $m >n$.  Piecewise affine functions provide a large and convenient set of \textit{test functions}. Indeed, for every $f \in C(K)$ and  $n \ge 0$ define $H_nf $ to be the unique $n$-piecewise affine function on $K$ that coincides with $f$ on $V_n$. By the construction of $H_nf$, it is clear that for every $n \ge 0$ and $w \in W_n$, we have for every $x \in K_w$,
\[
\inf_{K_w} f \le H_n f (x) \le \sup_{K_w} f.
\]
Since $f \in C(K)$, we deduce that $H_n f$ converges to $f$ uniformly on $K$. The following lemma is a useful property regarding $p$-energies for piecewise affine functions, see the proof of Theorem 5.8 in \cite{BC2020}. 
\begin{lem}\label{lem:p-energyPA}
Let $\Phi:K \to \mathbb{R}$ be an $n$-piecewise affine function. Then, for $1 \le p \le +\infty$, 
\ $\mathcal E_{p}^0 (\Phi) \le \cdots \le \mathcal E_{p}^n (\Phi)=\mathcal E_{p}^m(\Phi)$, where $m \ge n$, and $\mathcal E_{p} (\Phi)=\mathcal E_{p}^n (\Phi)$. In particular, $\Phi \in \mathcal{F}_p$ for every $1\le p \le +\infty $.
\end{lem}

\subsection{Characterizations of the  Sobolev spaces}

One of the major goals of the paper will be to prove the following theorem which follows from the combination of Theorem \ref{thm:charactW1p}, Theorem \ref{thm:charactW1infty}, Proposition \ref{prop:comparison1} and Proposition \ref{prop:comparison2}.

\begin{thm}\label{carac Sobolev}
Let $1 < p  \le +\infty $. For $f \in C(K)$ the following are equivalent:

\begin{list}{$(\theenumi)$}{\usecounter{enumi}\leftmargin=1cm \labelwidth=1cm\itemsep=0.2cm\topsep=.0 cm \renewcommand{\theenumi}{\arabic{enumi}}}
\item $f \in W^{1,p}(K)$;
\item $f \in \mathcal F_p$;
\item There exists  (a unique) $\partial f \in L^p( \mathcal S, \nu)$ such that for every $n \ge 0$ and adjacent vertices $u,v \in V_n$ with $u \le v$,
\[
f(v)-f(u)=\int_{\mathbf{e}(u,v)} \partial f \,  d\nu.
\]
\end{list}
Moreover, on $W^{1,p}(K)$, one has
\[
\| f \|_{W^{1,p}(K)} \simeq \| \partial f \|_{L^p(K,\nu)} = \| f \|_{\mathcal F_p}. 
\]

\end{thm}

We will actually obtain stronger  results, since the equivalence of the seminorms will be proved to be uniform over metric balls.  The case $p=1$ has to be treated separately and is covered in Theorem \ref{thm:charactW11} and Proposition \ref{prop:comparison2}.

\section{Weak gradients and Poincar\'e inequalities}

%For notation convenience, we will denote $\mathcal{F}_1=W^{1,1}(K)$ and $\mathcal{F}_\infty=W^{1,\infty}(K)$.
%
\subsection{Characterization of $\mathcal F_p$}

We first prove the following result:

\begin{thm}\label{thm:charactW1p}
Let $1 < p <+ \infty $. Let $f \in C(K)$. The following are equivalent:
\begin{list}{$(\theenumi)$}{\usecounter{enumi}\leftmargin=1cm \labelwidth=1cm\itemsep=0.2cm\topsep=.0 cm \renewcommand{\theenumi}{\arabic{enumi}}}
%\begin{enumerate}
\item $f \in \mathcal{F}_p$;
\item There exists $g \in L^p( \mathcal S, \nu)$ such that for every $n \ge 0$ and  for every adjacent $u,v \in V_n$ with $u \le v$,
\[
f(v)-f(u)=\int_{\mathbf{e}(u,v)} g \,  d\nu.
\]
%\end{enumerate}
\end{list}
Moreover, if $A \subset K$ is a  convex set, we have for every $f \in \mathcal F_p$, 
\[
\mathcal E_{A,p}(f)=\int_{A \cap \mathcal S} | g|^p d\nu.
\]
\end{thm}

\begin{proof}
We first prove that (2) implies (1). Indeed, it follows from (2) and H\"older's inequality that for every $n\ge0$ and every convex set $A \subset K$,
\begin{align*}
    3^{(p-1)n}\sum_{x,y \in V_n\cap A, x \sim y} |f(x)-f(y)|^p
    &\le 3^{(p-1)n}\sum_{x,y \in V_n \cap A, x \sim y} \left(\int_{\mathbf{e}(x,y)} |g| \,  d\nu\right)^p
    \\ &\le \sum_{x,y \in V_n \cap A, x \sim y} \int_{\mathbf{e}(x,y)} |g|^p \,  d\nu \le 2\int_{A \cap \mathcal{S}} |g|^p d\nu.
\end{align*}
Hence
 \[
\mathcal{E}_{A,p} (f)=\sup_n \mathcal{E}_{A,p}^n (f) \le \int_{A \cap \mathcal{S}} |g|^p d\nu
 \]
 and we deduce that $f\in W^{1,p}(K)$ with $\mathcal E_{A,p}(f)\le \int_{A \cap \mathcal S} | g|^p d\nu$.
 
It remains to show that (1) implies (2). If $\Phi$ is a piecewise affine function, it is clear that there exists a piecewise constant function, denoted by $\partial  \Phi$, such that for every adjacent $u,v \in V_n$ with $u \le v$,
\[
\Phi(v)-\Phi(u)=\int_{\mathbf{e}(u,v)} \partial  \Phi \,  d\nu.
\]

Consider then $f \in \mathcal F_p$. For every $n \ge 0$, we define $H_nf$ to be the unique $n$-piecewise affine function on $K$ that coincides with $f$ on $V_n$. We have then for every convex set $A \subset K$
 \[
 \sup_n \int_{A \cap \mathcal{S}} |\partial (H_nf)|^p d\nu=  \sup_n \frac{1}{2} \, 3^{(p-1)n} \sum_{x,y \in V_n\cap A, x \sim y} |f(x)-f(y)|^p <+\infty.
 \]
 The reflexivity of $L^p(\mathcal S,\nu)$ and Mazur lemma imply then that there exists a convex combination of a subsequence of $\partial (H_nf)$ that converges in $L^p(\mathcal S,\nu)$ to some  $g \in L^p(\mathcal S,\nu)$. Since $H_nf$ converges uniformly to $f$, we have then for every adjacent $u,v \in V_n$ with $u \le v$,
\[
f(v)-f(u)=\int_{\mathbf{e}(u,v)}  g\,  d\nu,
\]
and furthermore
\[
\int_{A \cap \mathcal S} | g|^p d\nu \le \sup_n \int_{A \cap \mathcal{S}} |\partial (H_nf)|^p d\nu\le \mathcal E_{A,p}(f).
\]
\end{proof}

We now turn to the case $p=+\infty$.

\begin{thm}\label{thm:charactW1infty}
Let $f \in C(K)$. The following are equivalent:
\begin{list}{$(\theenumi)$}{\usecounter{enumi}\leftmargin=1cm \labelwidth=1cm\itemsep=0.2cm\topsep=.0 cm \renewcommand{\theenumi}{\arabic{enumi}}}
%\begin{enumerate}
\item $f \in W^{1,\infty}(K)$;
\item $f \in \mathcal F_\infty$;
\item There exists $g \in L^\infty( \mathcal S, \nu)$ such that for every $n \ge 0$ and  for every adjacent $u,v \in V_n$ with $u \le v$,
\[
f(v)-f(u)=\int_{\mathbf{e}(u,v)} g \,  d\nu.
\]

%\end{enumerate}
\end{list}
Moreover, if $A \subset K$ is a  convex set, we have for every $f \in \mathcal F_\infty$, 
\[
\sup_{x,y \in A, x\neq y} \frac{|f(x)-f(y)|}{d(x,y)}=\mathcal E_{A,\infty}(f)=\| g \|_{L^\infty(A \cap \mathcal S,\nu)}.
\]

\end{thm}

\begin{proof}
We  begin with the proof that (3) implies (1). It follows from (3) that for every adjacent $u,v \in V_n$,
\[
| f (u)-f(v)| \le \| g \|_{L^\infty(\mathcal S, \nu)} \nu ( \mathbf{e}(u,v)) = \| g \|_{L^\infty(\mathcal S, \nu)} d(u,v).
\]
Using the tree structure of $V_n$ and the triangle inequality, this yields that for every $u,v \in V_n$,
\[
| f (u)-f(v)| \le \| g \|_{L^\infty(\mathcal S, \nu)} \nu ( \gamma_n(u,v)) =\| g \|_{L^\infty(\mathcal S, \nu)} d(u,v),
\]
where $\gamma_n(u,v)$ denotes the shortest path in $\bar V_n$ connecting $u$ and $v$.
Since $K$ is the closure of the skeleton $\mathcal S=\cup_n \bar V_n$ and $f$ is continuous, we deduce that $f$ is Lipschitz on $K$ and thus in $W^{1,\infty}(K)$ with $\| f \|_{W^{1,\infty}(K)} \le  \| g \|_{L^\infty(\mathcal S, \nu)}$.

We now prove that (1) implies (3). Let $u,v\in V_n$, $u \sim v$, $u \le v$.  If $f \in W^{1,\infty}(K)$, then its restriction to $\mathbf{e}(u,v)$ is Lipschitz continuous. Since $\mathbf{e}(u,v)$ is a compact interval and $\nu$ induces the Lebesgue measure of that interval, we deduce from well-known real analysis results (a weak version of Rademacher theorem) that there exists a function $g$ on $\mathbf{e}(u,v)$ which is $\nu$ essentially bounded  by $\| f \|_{W^{1,\infty}(K)}$  such that
\[
f(v)-f(u)=\int_{\mathbf{e}(u,v)} g \,  d\nu.
\]
Using a covering of $\bar{V}_n$ by its edges, we obtain a function $g$  defined on $\bar{V}_n$  such that for every $u,v\in V_n$, $u \sim v$, $u \le v$
\[
f(v)-f(u)=\int_{\mathbf{e}(u,v)} g \,  d\nu.
\]
Using the tree structure of $\bar{V}_n$, we see that this function $g$ is independent from $n$. 

Since the fact that (1) implies (2) with $\mathcal{E}_\infty (f) \le \| f \|_{W^{1,\infty}(K)}$  is obvious, we are left with the fact that (2) implies (1). We note that (2) implies that for every $x,y \in V_m$, $x \sim y$, 
\[
  |f(x)-f(y)| \le \mathcal{E}_\infty (f) \, d(x,y) .
\]
Using the tree structure of $V_m$ and triangle inequality, one gets that for every $x,y \in V_m$, 
\[
  |f(x)-f(y)| \le \mathcal{E}_\infty (f) \, d(x,y) .
\]
Using the density of $\cup_m V_m$ in $K$ and the continuity of $f$ finishes the proof that $f \in W^{1,\infty}(K)$ with $\| f \|_{W^{1,\infty}(K)} \le \mathcal{E}_\infty (f) $.

When $A \subset K$ is a convex set, the equality
\[
\sup_{x,y \in A, x\neq y} \frac{|f(x)-f(y)|}{d(x,y)}=\mathcal E_{A,\infty}(f)=\| g \|_{L^\infty(A \cap \mathcal S,\nu)}
\]
follows by similar arguments.
\end{proof}

For $p=1$ the situation is slightly different.

\begin{thm}\label{thm:charactW11}
 Let $f \in C(K)$. The following are equivalent:
\begin{list}{$(\theenumi)$}{\usecounter{enumi}\leftmargin=1cm \labelwidth=1cm\itemsep=0.2cm\topsep=.0 cm \renewcommand{\theenumi}{\arabic{enumi}}}
%\begin{enumerate}
\item $f \in \mathcal{F}_1$;
\item There exists a finite signed measure $\gamma_f$ on $\mathcal S$  such that for every $n$ and $u \in V_n$ $|\gamma_f | (\{ u \})=0$  and  for every adjacent $u,v \in V_n$ with $u \le v$,
\[
f(v)-f(u)=\gamma_f (\mathbf{e}(u,v)).
\]
%\end{enumerate}
\end{list}
Moreover, if $A \subset K$ is a  convex set, we have for every $f \in \mathcal F_1$, 
\[
\mathcal E_{A,1}(f)=|\gamma_f| (A \cap \mathcal S ).
\]
\end{thm}

\begin{proof}
Assume (2). In that case, for any convex set $A \subset K$
 \[
\mathcal{E}_{A,1}^m(f)=\frac{1}{2}  \sum_{x,y\in A \cap V_m, x \sim y} |f(x)-f(y)| \le \frac{1}{2}  \sum_{x,y\in A \cap V_m, x \sim y} |\gamma_f| ( \mathbf{e}(x,y))\le |\gamma_f| (A \cap \bar{V}_m ) .
\]
Therefore $f \in \mathcal F_1$ and $\mathcal E_{A,1}(f) \le |\gamma_f| (A \cap \mathcal S )$.
Assume (1) and fix $n \ge 0$, $u,v \in V_n$ be adjacent. From the triangle inequality we have for every $m \ge n$ and  $x,y \in \mathbf{e}(u,v) \cap V_m$
\[
|f(x)-f(y)|\le \mathcal{E}_{\mathbf{e}(x,y),1}^m(f)\le \mathcal{E}_{\mathbf{e}(x,y),1}(f).
\]
Similarly for every $N\ge 2$, and $x_i \in \mathbf{e}(u,v) \cap (\cup_{m\ge n}V_m)$, $1 \le i  \le N$, such that $x_1<\cdots< x_N$ we have
\[
\sum_{i=1}^{N-1} |f(x_{i+1})-f(x_i)| \le \mathcal{E}_{\mathbf{e}(u,v),1}(f).
\]
By continuity of $f$ and density of $\cup_{m\ge n}V_m$ in $K$ we deduce that
\[
\sum_{i=1}^{N-1} |f(x_{i+1})-f(x_i)| \le \mathcal{E}_{\mathbf{e}(u,v),1}(f).
\]
holds for every $N\ge 2$, and $x_i \in \mathbf{e}(u,v)$, $1 \le i  \le N$, such that $x_1<\cdots<x_N$. This means that the restriction of $f$ to the edge $\mathbf{e}(u,v)$ is a continuous bounded variation function. Therefore from a classical result in real analysis, there exist two non-decreasing continuous functions $f_1$ and $f_2$ on $\mathbf{e}(u,v)$ such that $f =f_1-f_2$ on $\mathbf{e}(u,v)$. We can then define a unique finite signed measure $\gamma_f$ on $\mathbf{e}(u,v)$ such that
\[
\gamma_f (\mathbf{e}(x,y))=(f_1(y)-f_1(x))-(f_2(y)-f_2(x)), \, x,y \in \mathbf{e}(u,v), x \le y .
\]
Note that
\[
|\gamma_f |(\mathbf{e}(x,y))= (f_1(y)-f_1(x))+(f_2(y)-f_2(x))=|f(x)-f(y)|.
\]
and that the $|\gamma_f|$ measure of a point is zero due to the continuity  of $f_1$ and $f_2$. Using the tree structure of $\bar V_n$ one obtains a finite measure $\gamma_f$ on $\mathcal{S}=\cup_n \bar V_n$ (and two continuous functions $f_1$ and $f_2$)  such that for every $n \ge 0$ and  every adjacent $u,v \in V_n$ with $u \le v$,
\[
f(v)-f(u)=\gamma_f (\mathbf{e}(u,v)).
\]
Moreover, if $A \subset K$ is convex we have for every $m \ge 0$
\[
|\gamma_f| (A \cap \bar{V}_m)\le \mathcal{E}_{A\cap \bar{V}_m ,1} (f) \le  \mathcal{E}_{A ,1} (f).
\]
This implies $|\gamma_f| (A \cap \mathcal S) \le  \mathcal{E}_{A ,1} (f)$ by letting $m\to \infty$.
\end{proof}

\subsection{Weak gradients}

\begin{defn}\label{defW11}
We define $W^{1,1}(K) \subset \mathcal F_1$ to be the set of $f \in C(K)$ such that there exists $g \in L^1( \mathcal S, \nu)$ such that for every adjacent $u,v \in V_n$ with $u \le v$,
\[
f(v)-f(u)=\int_{\mathbf{e}(u,v)} g \,  d\nu.
\] 
Such $g$ is then unique and the semi-norm on $W^{1,1}(K)$ is defined by
\[
\| f \|_{W^{1,1}(K)} =\int_\mathcal{S} | g | d\nu.
\]
\end{defn}

\begin{remark}
The inclusion $W^{1,1}(K) \subset \mathcal F_1$ is strict. Indeed,  consider a continuous function $f:[0,1]\to \mathbb R$ of bounded variation   which is not absolutely continuous with respect to the Lebesgue measure (like the so-called devil staircase). Consider now the unique continuous function $g$ on $K$ such that
\[
g(x,x)=f( \sqrt{2} x), \, x\in [0,\sqrt{2}/2]
\]
and such that $g$ is, for every $n$, constant on any connected component of $\bar{V}_n \setminus \mathbf{e}$ where $\mathbf{e}$ is the edge $( \sqrt{2} x, \sqrt{2}x), \, x\in [0,\sqrt{2}/2]$.  Then $g$ is in $\mathcal{F}_1$ but not $W^{1,1}(K)$.
\end{remark}

\begin{remark}
We use the notation $W^{1,1}(K)$ because that space appears as the endpoint of the real interpolation scale $W^{1,p}(K)$, $1<p\le+\infty$, see Theorem \ref{real inter sobo}.
\end{remark}

\begin{defn}
Let $1 \le  p \le +\infty$. For $f \in \mathcal{F}_{p}$ if $p >1$, and $f \in W^{1,1}(K)$ if $p=1$,  we will denote by $\partial f$ the unique function in $L^p(\mathcal S,\nu)$ such that for every $n$ and  for every adjacent $u,v \in V_n$ with $u \le v$
\[
f(v)-f(u)=\int_{\mathbf{e}(u,v)} \partial f \,  d\nu.
\] 
\end{defn}

\begin{remark}
It is easy to see that if $q_5=(0,0)$ is the center of $K$ and if $x \in V_m$, then for $f \in \mathcal{F}_{p}$
\[
f(x)-f(q_5)=\int_{\gamma_m(q_5,x)} \partial f \,  d\nu,
\]
where we recall that $\gamma_m(q_5,x)$ is the geodesic path in $\bar{V}_m$ connecting $q_5$ to $x$.
\end{remark}

\begin{remark}
The operator $\partial$ is defined modulo the orientation on $\mathcal S$ determined by the order $\le$ on pair of adjacent vertices in $V_m$. However  $| \partial f|$ is independent from the choice of orientation.
\end{remark}

\begin{remark}\label{sobolev cable system}
The set $\bar{V}_m$ is a cable system.  As such, see for instance Section 5.1 in  \cite{BK}, one can see any continuous function on $\bar{V}_m$ as a finite collection of functions $(f_\mathbb{e})_{\mathbf{e} \in E_m}$ where $E_m$ is the set of edges of $\bar{V}_m$ and $f_\mathbf{e}: \left[0, 3^{-m} \right] \to \mathbb{R}$ is a continuous functions (with the  appropriate boundary conditions). Then, it is easy to see that for $f \in \mathcal{F}_p$, $1 < p \le +\infty$ (or $f \in W^{1,1}(K)$ if $p=1$), denoting $f^m=f_{/\bar{V}_m}$, we have for all $\mathbf{e}$ in $E_m$, $f^m_\mathbf{e} \in W^{1,p}\left(\left[0, 3^{-m} \right]\right)$, where for an interval $I \subset \mathbb{R}$, $W^{1,p}(I)$ is the  usual $(1,p)$ Sobolev space. Note then that for $p=2$ our operator $\partial$ is similar to the exterior derivative considered in \cite{BK}. Thus, with the terminology of \cite{BK} one can see $L^p(\mathcal S, \nu)$, $1\le p \le+\infty$ as the set of $p$-integrable one-forms on $K$.
\end{remark}

\begin{remark}\label{scaling gradient}
Let $1 \le p \le +\infty$. It is clear that for $f \in \mathcal F_p$ for $p>1$ (or $f \in W^{1,1}(K)$ if $p=1$) and $w \in W_m$, we have $f \circ \Psi_w \in \mathcal F_p$ with
\[
\partial (f \circ \Psi_w)= 3^{-m} (\partial f  ) \circ \Psi_w.
\]
\end{remark}

\begin{remark}
For $p=2$, $\mathcal F_2$ is the domain of the standard self-similar Dirichlet form  $\mathcal E_2$ on $K$ and from the previous result one has
\[
\mathcal E_2 (f)=\int_{\mathcal S} | \partial f |^2 d\nu, \quad f \in \mathcal F_2.
\]
The measure $\nu$ is  a minimal energy dominant measure in the  sense of Hino \cite{MR2578475}.
\end{remark}

\subsection{Poincar\'e inequalities}

In this section we prove the Poincar\'e inequalities using Morrey type estimates.

\begin{thm}[Morrey type estimate]\label{thm:Morrey}
Let $A \subset K$ be a closed convex set. Let $1 \le p <+\infty$. For every $f \in \mathcal{F}_p$ and $x,y \in A$
\[
|f(x)-f(y)|^p \le d(x,y)^{p-1} \mathcal{E}_{A,p} (f).
\]
\end{thm}

\begin{proof}
We first assume $p >1$. Let $x,y \in (\cup_{n} V_n) \cap A$. We can find $m$ large enough so that $x,y \in V_m$. We have then from H\"older's inequality
\begin{align*}
|f(x)-f(y)| & \le  \int_{\gamma_m(x,y)} | \partial f | d\nu 
  \le d(x,y)^{1-\frac{1}{p}} \left( \int_{\gamma_m(x,y)} | \partial f |^p  d\nu\right)^{1/p} \\
 & \le d(x,y)^{1-\frac{1}{p}} \left( \int_{A \cap \mathcal S} | \partial f |^p  d\nu\right)^{1/p},
\end{align*}
where $\gamma_m(x,y)$ denotes the  geodesic path connecting $x$ and $y$ in $V_m$. 
Therefore, for every  $x,y \in (\cup_{n} V_n) \cap A$
\[
|f(x)-f(y)|^p \le d(x,y)^{p-1} \mathcal{E}_{A,p} (f).
\]
Since $\cup_{n} V_n$ is dense in $K$, the result follows by the continuity of $f$. For $p=1$ the proof is similar by using Theorem \ref{thm:charactW11} so the details are left to the reader.
\end{proof}

\begin{cor}\label{Poincare convex}
Let $A \subset K$ be a closed convex set. Let $1 \le p <+\infty$. For every $f \in \mathcal{F}_p$, and $x,y \in A$, there holds
\[
\fint_{A} \left|f(x)-\fint_Afd\mu\right|^p d\mu(x) 
\le \mathrm{diam}(A)^{p-1}\mathcal E_{A,p}(f).
\]
In particular, for any ball $B(x_0,r)\subset K$
\begin{align}\label{Poinc1}
\int_{B(x_0,r)} \left|f(x)-\fint_{B(x_0,r)}fd\mu\right|^p d\mu(x) 
\le C r^{p\alpha_p} \mathcal E_{B(x_0,r),p}(f).
\end{align}
\end{cor}
\begin{proof}
Applying H\"older's inequality and Theorem \ref{thm:Morrey}, we have
\begin{align*}
\int_A \left|f(x)-f_A\right|^p d\mu(x)
 & \leq 
\frac1{ \mu(A)^{p}} \int_{A} \left |\int_{A} (f(x)-f(y))d\mu(y)\right |^p d\mu(x)
\\ & \leq  
\frac1{ \mu(A)} \int_A \int_{A} |f(x)-f(y)|^p d\mu(y) d\mu(x)
\\ & \le
\mu(A) \mathrm{diam}(A)^{p-1}\mathcal E_{A,p}(f).
\end{align*}
The second inequality immediately follows from the $d_h$-Ahlfors regular property of $K$.
\end{proof}

\begin{remark}\label{sharp poinc}
It is worth noting that the exponent $p\alpha_p$ in the Poincar\'e inequality \eqref{Poinc1} is sharp. Indeed, consider the 0-piecewise affine function $f$ such that
\[
f(q_2)=f(q_4)=1, \, f(q_1)=f(q_3)=-1
\]
and $f(q_5)=0$. Then, by symmetry, for every $r>0$, $\int_{B(q_5,r)} f d\mu =0$. On the other hand for every $n \ge 0$,
\[
\int_{3^{-n}K} |f|^p d\mu=5^{-n}3^{-np}\int_{K} |f|^p d\mu
\]
and
\[
\int_{3^{-n}K\cap \mathcal S} |\partial f|^p d\nu=3^{-n} \int_{ \mathcal S} |\partial f|^p d\nu.
\]
Therefore, for $r=3^{-n}$, we have $\int_{B(q_5,r)} |f|^p d\mu \simeq r^{p\alpha_p} \int_{B(q_5,r) \cap \mathcal S} |\partial f|^p d\nu$ when $r \to 0$.
\end{remark}

\paragraph{Proving Poincar\'e inequalities using discrete approximations}

\

\

\noindent It is possible to give a second proof of Theorem \ref{thm:Morrey} and thus of Corollary \ref{Poincare convex} using  discrete approximations on $V_m$ as in  \cite{Chen} and then taking the limit when $m \to \infty$. Such an approach would be  natural in the context  of more general nested fractals. For completeness, we sketch  the argument (mostly adapted from \cite{Chen}). 

Let $A$ be a closed and convex set and $f \in \mathcal{F}_p$, $1 \le p <+\infty$. For any edge $e$ in $V_m$, denote by $e_+$ and $e_-$ its two vertices. Then, for $x,y \in A \cap V_m$,
\begin{equation}\label{eq:chain}
|f(x)-f(y)|\le \sum_{e\in \gamma_m(x,y)}|f(e_+)-f(e_-))|,
\end{equation}
where $\gamma_m(x,y)$ is the geodesic path connecting $x$ and $y$ in $V_m$. 
In addition, denote by $|\gamma_m(x,y)|$ the number of edges in $\bar{V}_m$ for the path $\gamma_m(x,y)$. Then we note that from the structure of the Vicsek set, 
\begin{equation}\label{eq:chainNo}
|\gamma_m(x,y)|= 3^m d(x,y).
\end{equation}
The above estimate and H\"older's inequality give that
\begin{align*}
|f(x)-f(y)| 
& \leq  \sum_{e \in \gamma_m(x,y)}|f(e_+)-f(e_-)|
 \leq  
|\gamma_m(x,y)|^{1-\frac1p} \Bigg(\sum_{e \in \gamma_m(x,y)}|f(e_+)-f(e_-)|^p \Bigg)^{\frac1p}
\\ & \le 
d(x,y)^{1-\frac1p} \Bigg(3^{m(1-p)}\sum_{v,w\in A \cap V_m, v\sim w}|f(v)-f(w)|^p\Bigg)^{\frac1p}
\\ &=d(x,y)^{1-\frac1p} \mathcal{E}_{A,p}^m(f)^{\frac1p}.
\end{align*}
Now, for general $x,y \in A$, we pick  sequences $x_m,y_m \in V_m$ such that $x_m \to x$ and $y_m \to y$ and let $m\to +\infty $ in the previous inequality thanks to the continuity of $f$.

%------------------------------------------------------------------------------------------------------------------------------------------%

\section{Korevaar-Schoen-Sobolev and Haj\l{}asz-Sobolev spaces}

\subsection{Comparison of the discrete and Korevaar-Schoen $p$-energies}

In this section, we  compare the $L^p$ Korevaar-Schoen energy (see Definition \ref{def:p-KSS}) and the $p$-energy defined from the limit approximation of discrete $p$-energy (see Definition \ref{def:p-energyVS}). 

\begin{prop}\label{prop:comparison1}
Let  $1<p <+\infty$.  There exist constants $c,C>0$ such that for every $f \in C(K)$, $x_0 \in K$, and $r >0$
\[
\mathcal{E}_{B(x_0,r),p} (f)  \le C  \| f \|^p_{W^{1,p}(B(x_0,cr))}.
\]
In particular, if $f \in W^{1,p}(K)$ then $\mathcal{E}_{p} (f)<+\infty$ and thus $f \in \mathcal F_p$.
\end{prop}
\begin{proof}
We  use a strategy found in the proof of \cite[pages 108-110]{HuJiWen}. The method in that paper deals with the Sierpinski gasket, but it can be applied as well to the Vicsek set modulo appropriate modifications.  For a fixed ball $B(x_0,r)\subset K$ with $r \le 2$, let $n_0 \ge 0$ be such that $ 2 \cdot 3^{-n_0-1} < r \le 2 \cdot 3^{-n_0}$. From now on we assume that  $m>n_0$. Notice that for any $x,y \in V_m$ which are neighbors, there exists a unique $m$-simplex $K_w$ such that $x,y\in K_w$. In this case, we also have $x,y\in V_m^w$. By the basic convexity inequality, 
\[
|f(x)-f(y)|^p \le 2^{p-1} \left( |f(x)-f(u)|^p+|f(u)-f(y)|^p\right)
\]
one has 
\begin{equation}\label{eq:convex}
|f(x)-f(y)|^p\le \frac{2^{p-1}}{\mu(K_w)}\int_{K_w} \left(|f(x)-f(u)|^p+|f(u)-f(y)|^p\right) d\mu(u).
\end{equation}
In order to estimate $\mathcal{E}_{B(x_0,r),p}^m(f)$, we denote 
\[
\mathcal I_m=\{w\in W_m: \, \exists \, x,y\in  V^w_m \cap B(x_0,r)  \text{ such that } x \sim y\}.
\]
Observe that there exists a constant $c>1$ ($c=2$ will do) such that $\cup_{w\in \mathcal I_m} K_w \subset B(x_0,cr)$.
By \eqref{eq:convex}, one has
\[
\mathcal{E}_{B(x_0,r),p}^m(f) 
\le C  3^{m(p-1)}\sum_{w\in \mathcal I_m}\sum_{x\in V_m^w}\frac{1}{\mu(K_w)}\int_{K_w} |f(x)-f(u)|^p d\mu(u).
\]

Now let $x\in V_m^w$ be fixed. There exists a sequence of sets $\{S_j\}_{j\ge 0}$ which shrinks to $x$ and where $S_j$ is an $(m+j)$-simplex. Indeed, take $i_0\in W$ such that $q_{i_0}\in V_0$ is the vertex satisfying $x=\Psi_w(q_{i_0})$. We set 
%can take $i_0\in W$ such that $\psi_{i_0}(p_0)=p_0$ and set 
\[
S_0=K_w, \quad S_1=\Psi_{w}\circ \psi_{i_0}(K), \quad \cdots, \quad S_j=\Psi_{w}\circ \underbrace{\psi_{i_0}\circ \cdots \circ \psi_{i_0}}_\text{ j times}(K).
\]
Then one observes that $x\in S_j$ for every $j\ge 0$ and that the sequence $\{S_j\}_{j\ge 0}$ shrinks to the vertex $x$. Now for every $u_0:=u\in S_0$, $u_j\in S_j$ for $j>0$ and $l\ge1$,  we have that 
\begin{align*}
|f(x)-f(u)|^p &\le 2^{p-1}\left(|f(x)-f(u_l)|^p+|f(u_l)-f(u)|^p\right)
\\ &\le 2^{p-1}|f(x)-f(u_l)|^p+\sum_{j=1}^l 2^{(p-1)(j+1)} |f(u_j)-f(u_{j-1})|^p.
\end{align*}
Integrating the above inequality with respect to each $u_j\in S_j$ ($0\le j\le l$) and dividing by $\mu(S_0)\mu(S_1)\cdots \mu(S_l)$, we then obtain
\begin{multline*}
\frac{1}{\mu(K_w)}\int_{K_w} |f(x)-f(u)|^p d\mu(u) 
\le \frac{2^{p-1}}{\mu(S_l)}\int_{S_l} |f(x)-f(u_l)|^pd\mu(u_l)
\\+\sum_{j=1}^{l}\frac{2^{(p-1)(j+1)}}{\mu(S_{j-1})\mu(S_j)}\int_{S_{j-1}}\int_{S_j} |f(u_j)-f(u_{j-1})|^p d\mu(u_j)d\mu(u_{j-1}).
\end{multline*}
Since $f$ is continuous, the first term on the right hand side tends to zero as $l\to \infty$.
Next we note that $\mu(S_j)=5^{-(m+j)}$ and $d(u_j,u_{j-1})\le 2 \cdot 3^{-(m+j-1)}$ for any $u_{j-1}\in S_{j-1}, u_j\in S_j$, then for $1\le j\le l$ there holds
\begin{align*}
&\frac1{\mu(S_{j-1})\mu(S_j)}\int_{S_{j-1}}\int_{S_j} |f(u_j)-f(u_{j-1})|^p d\mu(u_j)d\mu(u_{j-1})
\\ & \quad\quad\le 5^{2m+2j-1}\int_{S_0}\int_{B(u,2 \cdot 3^{-(m+j-1)})} |f(u)-f(v)|^p d\mu(v)d\mu(u).
\end{align*}
Also, one always has $S_j\subset K_w\subset B(x_0,cr)$ for any $j\ge 0$. Therefore the second term is bounded above  by 
\begin{align*}
\sum_{j=1}^{l}2^{(p-1)(j+1)} 5^{2m+2j-1}\int_{K_w}\int_{B(u,2 \cdot 3^{-(m+j-1)})\cap B(x_0,cr)} |f(u)-f(v)|^p d\mu(v)d\mu(u).
\end{align*}
Summing up the integral above over all $w\in \mathcal I_m$ and letting $l\to \infty$, we have then 
\begin{align*}
    \mathcal{E}_{B(x_0,r),p}^m(f) \le &C 3^{m(p-1)}  \sum_{j=1}^{\infty}2^{(p-1)(j+1)} 5^{2m+2j-1}
    \\ &\int_{B(x_0,cr)}\int_{B(u,2 \cdot 3^{-(m+j-1)})\cap B(x_0,cr)} |f(u)-f(v)|^p d\mu(v)d\mu(u)
    \\ \le & C  \sum_{j=1}^{\infty}2^{(p-1)(j+1)} 3^{-(p-1)(j-1)} %3^{-(m+j-1)(p\alpha_p+d_h)}
    \\ &\frac{1}{3^{-(m+j-1)(p\alpha_p+d_h)}}\int_{B(x_0,cr)}\int_{B(u,2 \cdot 3^{-(m+j-1)})\cap B(x_0,cr)} |f(u)-f(v)|^pd\mu(v)d\mu(u)
    \\ \le & C \sup_{\rho \in (0,2 \cdot 3^{-m})} \frac{1}{\rho^{p\alpha_p+d_h}}\int_{B(x_0,cr)}\int_{B(u,\rho)\cap B(x_0,cr)} |f(u)-f(v)|^pd\mu(v)d\mu(u)
\end{align*}
where the second inequality follows from the fact that $p\alpha_p+d_h=p-1+2d_h$.  In view of \eqref{eq:lim}, we thus conclude the proof by taking $\lim_{m\to \infty}$.
\end{proof}

As an immediate corollary we obtain from Corollary \ref{Poincare convex} the $L^p$-Poincar\'e inequality in the Korevaar-Schoen-Sobolev spaces.

\begin{cor}
Let $p  >1 $. Then there exist constants $c,C>0$ such that for any $f \in W^{1,p}(K)$,  $x_0\in K$ and $r>0$ we have:
\[
\int_{B(x_0,r)} |f(x)-f_{B(x_0,r)}|^p d\mu(x) 
\le Cr^{p-1+d_h} \| f \|_{W^{1,p}(B(x_0,cr)}^p.
\]
\end{cor}

\begin{remark}
For the Vicsek set,  $L^p$-Poincar\'e inequalities in the Korevaar-Schoen-Sobolev spaces were obtained in \cite{BC2020} for the range $1 \le p \le 2$. The inequalities in \cite{BC2020} are actually stronger since we used on the right hand side the functional
\[
\liminf_{r\to 0^+} \frac{1}{r^{p\alpha_p}} \int_A\int_{B(x,r) \cap A}\frac{|f(y)-f(x)|^p}{ \mu(B(x,r))}\, d\mu(y)\, d\mu(x)
\]
instead of $\| f \|_{W^{1,p}(A)}$ (which is defined using a $\limsup$). However, the techniques in \cite{BC2020} do not apply for $p \ge 2$.
\end{remark}

For the comparison of the reverse direction, we have in fact the following stronger statement.

\begin{prop}\label{prop:comparison2}
Let  $1 \le p<\infty$.  There exists constants $c,C>0$ such that for every $f \in \mathcal F_p$, $x_0 \in K$, and $r >0$
\[
\sup_{R>0}\frac{1}{R^{p\alpha_p}}\int_{B(x_0,r)}\int_{B(x_0,r)\cap B(x,R)}\frac{|f(x)-f(y)|^p}{\mu(B(x,R))} d\mu(y) d\mu(x) \le C \mathcal{E}_{B(x_0,cr),p}(f).
\]
In particular, for $1<p<+\infty$,  $\mathcal F_p \subset W^{1,p}(K)$.
\end{prop}
\begin{proof}
Without loss of generality, we take $r\le 2$.
We first assume $ R  \ge r/6$. Let $f\in \mathcal F_p$,  then
\begin{align*}
\int_{B(x_0,r)}\int_{B(x_0,r)\cap B(x,R)}\frac{|f(x)-f(y)|^p}{\mu(B(x,R))} d\mu(y) d\mu(x)
\le &    \int_{B(x_0,r)}\int_{B(x_0,r)}\frac{|f(x)-f(y)|^p}{\mu(B(x,R))} d\mu(y) d\mu(x).
\end{align*}
From   Theorem \ref{thm:Morrey}, we have $|f(x)-f(y)|^p \le  C r^{p-1} \mathcal E_{ B(x_0,r),p}(f) $. Therefore,
\begin{align}\label{eq:tyh}
 \int_{B(x_0,r)}\int_{B(x_0,r)}\frac{|f(x)-f(y)|^p}{\mu(B(x,R))} d\mu(y) d\mu(x) \le C r^{p-1+d_h} \mathcal E_{ B(x_0,r),p}(f)\le C R^{p\alpha_p} \mathcal E_{ B(x_0,r),p}(f),
\end{align}
and 
\begin{equation}\label{eq:bigR}
\int_{B(x_0,r)} \int_{B(x,R) \cap B(x_0,r)}\left|f(x)-f(y) \right|^p d\mu(y) d\mu(x)
\le C R^{p \alpha_p+d_h} \mathcal E_{B(x_0,r),p} (f).
\end{equation}

We then assume $0 < R  \le r/6$. 
%Some of the ideas are adapted from the proof of \cite[Theorem 5.8]{BC2020}.
Let $k$ be the unique integer such that  
\[
3^{-(k+1)} <R\le 3^{-k}.
\] 
Consider the covering of $B(x_0,r)$ by a collection of $k$-simplices $\{K_{w}\}_{w\in \mathcal A_k}$, where 
\[
\mathcal A_k:=\{w\in W_k:K_w\cap B(x_0,r)\ne \emptyset\}.
\]
%Let $\Phi$ be an $n$-piecewise affine function on $K$ . 
%such that $\Phi(x)=0$ for any vertex $x\in \cup_{w\in \mathcal A_k}K_w\setminus B(x_0,r)$.
Notice that for $x\in K_{w}$, we have that $B(x,R)\subset K_{w}^{*}$, where $K_w^*$ denotes the union of $K_w$ and all its adjacent $n$-simplices. 
Then 
%Theorem \ref{thm:Morrey}  gives
\begin{align*}
&\int_{B(x_0,r)} \int_{B(x,R) \cap B(x_0,r)}\left| f(x)-f(y) \right|^p d\mu(y) d\mu(x)
\\ &\qquad\le
\sum_{w \in \mathcal A_k} \int_{K_{w}} \int_{B(x,R) \cap B(x_0,r)}\left| f(x)-f(y) \right|^p d\mu(y) d\mu(x) 
\\ 
&\qquad \le \sum_{w \in \mathcal A_k}  \int_{K_{w}} \int_{K_{w}^{*} \cap B(x_0,r)}\left| f(x)-f(y) \right|^p d\mu(y) d\mu(x). 
%\\
%&\qquad \le  C  \sum_{w \in \mathcal A_k}  \int_{K_{w}} \int_{K_{w}^{*} \cap  B(x_0,r)} d\mu(y) d\mu(x)\mathcal E_{K_w^*,p}^n( \Phi) 
%\\
%&\qquad \le C 5^{-2k} \sum_{v \in \mathcal B_k}  \mathcal E_p^n( \Phi \circ \Psi_{v}),
\end{align*}
For any $x\in K_w$ and $y\in K_{w}^{*} \cap  B(x_0,r)$, Theorem \ref{thm:Morrey}  gives
\[
\left| f(x)-f(y) \right|\le C 3^{-k(p-1)}\mathcal  E_{K_w^*,p}(f) .
\]
We also observe the following two facts:
\begin{itemize}
\item There exists a constant $c>1$ such that for any $w\in \mathcal A_k$, $K_{w}^{*} \subset B(x_0, cr)$;
\item The family $\{K_{w}^{*}\}_{w\in \mathcal A_k}$ has bounded overlap property.
\end{itemize}
Hence
\begin{align*}
&\sum_{w \in \mathcal A_k}  \int_{K_{w}} \int_{K_{w}^{*} \cap B(x_0,r)}\left| f(x)-f(y) \right|^p d\mu(y) d\mu(x) 
\\
&\qquad \le  C  \sum_{w \in \mathcal A_k}  \int_{K_{w}} \int_{K_{w}^{*} \cap  B(x_0,r)} 3^{-k(p-1)}\mathcal E_{K_w^*,p}( f) d\mu(y) d\mu(x) 
\\
&\qquad \le C 5^{-2k} 3^{-k(p-1)}\sum_{v \in \mathcal A_k}  \mathcal E_{K_w^*,p}(f)
\le C R^{p-1+2d_h}  \mathcal E_{B(x_0,cr),p}(f),
\end{align*}
and 
\begin{equation}\label{eq:smallR}
\int_{B(x_0,r)} \int_{B(x,R) \cap B(x_0,r)}\left|f(x)-f(y) \right|^p d\mu(y) d\mu(x)
\le C R^{p \alpha_p+d_h} \mathcal E_{B(x_0,cr),p} (f).
\end{equation}

We conclude from \eqref{eq:bigR} and \eqref{eq:smallR} that for $f\in \mathcal F_p$ and $R>0$ 
\begin{align*}
\frac{1}{R^{p \alpha_p+d_h}}\int_{B(x_0,r)} \int_{B(x,R) \cap B(x_0,r)}\left|f(x)-f(y) \right|^p d\mu(y) d\mu(x)
\le C  \mathcal E_{B(x_0,cr),p} (f)
\end{align*}
and the proof is finished by taking $\sup_{R>0}$ in the left side.
\end{proof}

As a consequence of Propositions \ref{prop:comparison1} and \ref{prop:comparison2}, we record the following estimate which will be a key ingredient in a next section regarding the real interpolation of the Besov spaces.

\begin{cor}\label{cor:sup-limsup}
Let $1 < p <+\infty$. There exists a constant $C>0$ such that for every $f \in W^{1,p}(K)
$
\begin{equation}\label{eq:sup-limsup}
  \sup_{r >0} \frac{1}{r^{p\alpha_p}} \int_K\int_{B(x,r) }\frac{|f(y)-f(x)|^p}{ \mu(B(x,r))}\, d\mu(y)\, d\mu(x) 
 \le  C \| f \|^p_{W^{1,p}(K)}.
\end{equation}
\end{cor}

\subsection{Maximal functions and triviality of the Haj\l{}asz-Sobolev spaces}

Let $p > 1$. For $f \in W^{1,p}(X)$  we introduce the following maximal  function

\begin{equation}\label{eq:fnt-g}
g_f(x):=\sup_{r> 0}\frac{1}{\mu(B(x,r))^{1/p}} \left(\int_{B(x,r)} | \partial f|^p d\nu \right)^{1/p}.
\end{equation}

As in \cite{BC2020} or \cite{MR3125121} it is easy to see that for $p>1$ the maximal function  $g_f$ is weak $L^p(K,\mu)$ bounded and that the Poincar\'e inequality in Corollary \ref{Poincare convex} implies the following Lusin-H\"older estimate:

\begin{prop}
Let $p >1$. Then there exists a constant $C$ such that for every $f \in W^{1,p}(X)$,
\begin{align}\label{lusin-holder}
|f(x)-f(y)|\le C d(x,y)^{\alpha_p}(g_f(x)+g_f(y)).
\end{align}
\end{prop}

\begin{remark}
It is interesting to note that the estimate \eqref{lusin-holder} implies (and is therefore equivalent to) the Poincar\'e inequality on balls in  Corollary \ref{Poincare convex}. This can be proved as in the proof of Lemma 5.15 in \cite{MR1654771}. We thank an anonymous referee for this remark.
\end{remark}

The following proposition shows that the maximal function $g_f$ can not be in $L^p(X,\mu)$ unless $f$ is constant.

\begin{prop}\label{triviality}
Let $p >1 $. Let $f \in C(K)$. If there exists $g \in L^p(K,\mu)$ such that  $\mu$-almost everywhere
\[
|f(x)-f(y)|\le  d(x,y)^{\alpha_p}(g(x)+g(y)),
\]
then $f$ is constant.
\end{prop}

\begin{proof}
We first obtain that for every $w\in W_n$ 
\[
|f(\Psi_w(x))-f(\Psi_w(y))|\le  3^{-n\alpha_p } d(x,y)^{\alpha_p}(g(\Psi_w(x))+g(\Psi_w(y))),
\]

Then, 
\begin{align*}
\|f\circ \Psi_w\|_{W^{1,p}(K)} & =\limsup_{r\to 0^+} \frac{1}{r^{\alpha_p}}\Bigg( \int_K\int_{B(x,r)}\frac{|f(\Psi_w(y))-f(\Psi_w(x))|^p}{ \mu(B(x,r))}\, d\mu(y)\, d\mu(x)\Bigg)^{\frac1p} \\
 &\le C 3^{-n\alpha_p }\limsup_{r\to 0^+} \Bigg( \int_K\int_{B(x,r)}\frac{|g(\Psi_w(y))|^p+|g(\Psi_w(x))|^p}{ \mu(B(x,r))}\, d\mu(y)\, d\mu(x)\Bigg)^{\frac1p}  \\
 &\le C 3^{-n\alpha_p } \left(\int_{K} (g\circ \Psi_w)^p d\mu\right)^{1/p}
\end{align*}

From Theorem \ref{carac Sobolev} we get that for every $w\in W_n$
\[
\int_{\mathcal S} |\partial (f \circ \Psi_w) |^p d\nu \le C 3^{-np\alpha_p } \int_{K} (g \circ \Psi_w)^p d\mu.
\]
From Remark \ref{scaling gradient} this yields
\[
3^{-np} 3^n \int_{\mathcal S \cap K_w} |\partial f  |^p d\nu \le C 3^{-np\alpha_p} 3^{nd_h} \int_{K_w} g^p d\mu.
\]
We obtain therefore that for every simplex $K_w$
\[
\int_{\mathcal S \cap K_w} |\partial f  |^p d\nu \le C  \int_{K_w} g^p d\mu
\]

Consider then an edge $\mathbf{e}(u,v)$, $u,v \in V_n$, $u \sim v$. For $m \ge n$, one can cover this edge with a union $A_m$ of $N_m$ $m$-simplices with $N_m \le 3^{m-n}$. One has then
\[
\int_{\mathbf{e}(u,v)} |\partial f |^p d\nu \le C  \int_{A_m} g^p d\mu
\] 
Since $\mu (A_m)\le N_m 5^{-m}$ goes to zero when $m\to +\infty$, one obtains $\int_{\mathbf{e}(u,v)} |\partial f |^p d\nu=0$. 
Since it is true for every edge $\mathbf{e}(u,v)$, we deduce that $\nu$ almost everywhere $\partial f=0$ and thus that $f$ is constant.
\end{proof}

\section{Real interpolation theory of the Besov-Lipschitz and Sobolev spaces}

\subsection{Basics of the $K$ method for real interpolation}\label{subSec:interpolation}

In this section, mostly to fix notations, we recall some basic definitions and results of the $K$ method for real interpolation. Those definitions are mostly taken from \cite[Section 2]{GKS}.
For details, we refer  for instance to \cite[Chapters 3 and 5]{BS}.  In the following we will use the interpolation theory for seminormed spaces as in  \cite{GKS}.

Let $X_0$ and $X_1$ be two Banach spaces. Assume that the pair $(X_0,X_1)$ is a compatible couple, i.e., there is some Hausdorff topological vector space in which each of $X_0$ and $X_1$ is continuously embedded. Then the sum $X_0+X_1$ is a Banach space under the norm 
\[
\|f\|_{X_0+X_1}:=\inf\{\|f_0\|_{X_0}+\|f_1\|_{X_1}, f=f_0+f_1\}.
\]

The $K$-functional of $(X_0,X_1)$ is defined for each $f\in X_0+X_1$ and $t>0$ by 
\[
K(f,t,X_0,X_1):=\inf\{\|f_0\|_{X_0}+t\|f_1\|_{X_1}, f=f_0+f_1\}.
\]

Suppose that $0<\theta<1$, $1\le q<\infty$ or $0\le \theta\le 1$, $ q=\infty$. Then the interpolation space $(X_0,X_1)_{\theta,q}$ consists of functions $f\in X_0+X_1$ such that 
\[
\|f\|_{\theta,q}=
\begin{cases}
\left(\int_0^{\infty}(t^{-\theta}K(f,t,X_0,X_1))^q\frac{dt}{t}\right)^{1/q}, &0<\theta<1,1\le q<\infty, \\
\sup_{t>0} t^{-\theta}K(f,t,X_0,X_1), &0\le \theta\le 1, q=\infty,
\end{cases}
\]
is finite. In that context, the reiteration theorem (see \cite[Chapter 5, Theorem 2.4]{BS}) writes as follows:
\begin{thm}\label{reiteration}
Let $(X_0,X_1)$ be a compatible couple and suppose $0\le \theta_0<\theta_1\le1$. Let $\overline{X}_{\theta_j}$ be an intermediate space of class $\theta_j$, $j=0,1$. Then for $0<\theta<1$ and $1\le q\le \infty$, one has $(\overline{X}_{\theta_0},\overline{X}_{\theta_1})_{\theta,q}=(X_0,X_1)_{\theta',q}$, where $\theta'=(1-\theta)\theta_0+\theta \theta_1$.
\end{thm}

\subsection{Besov-Lipschitz spaces}

We  consider the Besov Lipschitz spaces that were studied in \cite{BV3,BV1}, see also \cite{Gri}.
\begin{defn}\label{def:BesovLip}
For $p \ge 1$ and $ \alpha>0$,  the Besov Lipschitz space $\mathcal{B}^{\alpha}_{p,\infty}(K)$ is defined by
\[
\mathcal{B}^{\alpha}_{p,\infty}(K)=\left\{ f \in L^p(K,\mu), \,  \sup_{r >0} \frac{1}{r^{\alpha}}\Bigg( \int_K\int_{B(x,r)}\frac{|f(y)-f(x)|^p}{ \mu(B(x,r))}\, d\mu(y)\, d\mu(x)\Bigg)^{\frac1p} <+\infty \right\}.
\]
\end{defn}
We note that by definition, for $p>1$, $\mathcal{B}^{\alpha_p}_{p,\infty}(K)=W^{1,p}(K)$. 
%
%It follows from Ahlfors regularity and compactness of $K$ that $\mathcal{B}^{\alpha}_{p,\infty}(K)$ is for every $\beta>0$ also equal to
%\begin{align}\label{Lp local}
%\left\{ f \in L^p(K,\mu), \,  \sup_{r\in (0,\beta)} \frac{1}{r^{\alpha}}\Bigg( \int_K\int_{B(x,r)}\frac{|f(y)-f(x)|^p}{ \mu(B(x,r))}\, d\mu(y)\, d\mu(x)\Bigg)^{\frac1p} <+\infty \right\}
%\end{align}
%with comparable seminorms.
%It  follows from \cite[Theorem 2.4]{BV3} that $\mathcal{B}^{\alpha}_{p,\infty}(K)$ is also equal (with equivalent seminorms norms) to the heat semigroup based Besov space $\mathbf{B}^{p,\alpha/d_w}(K)$ that was introduced in \cite{BV1}:
%\[
%\mathbf{B}^{p,\alpha/d_w}(K)=
%\left\{ f \in L^p(K,\mu), \,  \sup_{t \in (0,1)} \frac{1}{t^{\frac\alpha{d_w}}}\Bigg( \int_K\int_{ K}|f(y)-f(x)|^pp_t(x,y)\, d\mu(y)\, d\mu(x)\Bigg)^{\frac1p} <+\infty \right\},
%\]
%where $p_t(x,y)$ is the heat kernel on $K$.
%

\subsection{Interpolation of the Besov-Lipschitz spaces, $p>1$}

The goal of this section is to prove the following theorem:

\begin{thm}\label{interpolation}
 For every $p > 1$ and $0\le \alpha \le\alpha_p=1-\frac1p+\frac{d_h}{p} $
\[
\mathcal{B}^{\alpha}_{p,\infty}(K)=(L^p(K,\mu), W^{1,p}(K))_{\alpha/\alpha_p,\infty}.
\]
\end{thm}

The key ingredient to prove this interpolation result is the following estimate that follows from our previous results (see Corollary \ref{cor:sup-limsup}):
\begin{align*}
 & \sup_{r >0} \frac{1}{r^{\alpha_p}}\Bigg( \int_K\int_{B(x,r) }\frac{|f(y)-f(x)|^p}{ \mu(B(x,r))}\, d\mu(y)\, d\mu(x)\Bigg)^{\frac1p} \\
 \le & C \limsup_{r\to 0^+} \frac{1}{r^{\alpha_p}}\Bigg( \int_K\int_{B(x,r) }\frac{|f(y)-f(x)|^p}{ \mu(B(x,r))}\, d\mu(y)\, d\mu(x)\Bigg)^{\frac1p}.
\end{align*}
We note that this estimate implies that for $\alpha>\alpha_p$ the space $\mathcal{B}^{\alpha}_{p,\infty}(K)$ is trivial, i.e., $\mathcal{B}^{\alpha}_{p,\infty}(K)$ only consists of constant functions. Therefore the interpolation scale given by Theorem \ref{interpolation} is optimal with the endpoints $L^p(K,\mu)$ and $W^{1,p}(K)$.

Following the notation in Section \ref{subSec:interpolation}, the $K$-functional of the couple $(L^p(K,\mu), W^{1,p}(K))$ is defined for $f \in L^p(K,\mu)+ W^{1,p}(K)$ by
\[
K(f,t)=\inf \{ \|g \|_{L^p(K,\mu)}+t \| h \|_{W^{1,p}(K)}, \, f=g+h \}.
\]
For any $0\le \theta\le 1$, the interpolation space $(L^p(K,\mu), W^{1,p}(K))_{\theta,\infty}$ consists of all  $f \in L^p(K,\mu)+ W^{1,p}(K)$ such that $\sup_{t>0} t^{-\theta}K(f,t)<\infty$.

 For simplicity, we adopt the notation $E_p(f,r)$ for $f\in L^p(K,\mu)$ and  $r>0$  as in \cite{GKS}, that is, 
\[
E_p(f,r)= \int_K\int_{B(x,r) }\frac{|f(y)-f(x)|^p}{ \mu(B(x,r))}\, d\mu(y)\, d\mu(x).
\]
Adapting to our framework techniques from \cite[Theorem 4.1]{GKS}, we obtain the following main result of this section.
\begin{thm}\label{thm:K-func}
Let $p>1 $. There exist  $C_1, C_2>0$   such that for any $f\in L^p(K,\mu)+ W^{1,p}(K)$ and $r>0$,
\[
C_1 E_p(f,r)^{\frac1p} \le K(f,r^{\alpha_p}) \le C_2 E_p(f,r)^{\frac1p}.
\]
\end{thm}

\begin{proof}
It is easy to show the inequality
\[
C_1  E_p(f,r)^{\frac1p}  \le K(f,r^{\alpha_p}).
\]
 Indeed, suppose that  $f=g+h$, where $g\in L^p(K,\mu)$ and $h\in W^{1,p}(K)$.
Then by Minkowski's inequality and Corollary \ref{cor:sup-limsup}, we obtain
\begin{align*}
     E_p(f,r)^{\frac1p}  
     & \le E_p(g,r)^{\frac1p}+ E_p(h,r)^{\frac1p}
     \le C\left(\|g\|_{L^p(K,\mu)}+r^{\alpha_p}r^{-\alpha_p}E_p(h,r)^{\frac1p}\right)
     \\ &\le  C\left(\|g\|_{L^p(K,\mu)}+r^{\alpha_p}\|h\|_{W^{1,p}(K)}\right).
\end{align*}

Now turn to the proof of the second inequality, that is, $K(f,r^{\alpha_p}) \le C_2  E_p(f,r)^{\frac1p}$.
Given a function $f \in L^p(K,\mu)$, we first define a sequence of piecewise affine functions $\{\Phi_n\}_{n\ge 1}$ associated with $f$ on the cable systems $\{\bar{V}_n\}_{n\ge 1}$ as follows. 

For any fixed $n\ge 1$, we define the function $f_n$ on $V_n$ by
\[
f_{n}(v):=\frac{1}{\mu(K_{n+1}^*(v))} \int_{K_{n+1}^*(v) } f d\mu,\quad v \in V_n,
\]
where $K_{n+1}^*(v)$ is the union of the $(n+1)$-simplices containing $v$. Then, let $\Phi_n$ be the unique piecewise affine function that coincides with $f_n$ on $V_n$. More precisely, one writes
\[
\Phi_n(x)=\sum_{v \in V_n} \left(\frac{1}{\mu(K_{n+1}^*(v))} \int_{K_{n+1}^*(v) } f d\mu \right) \, u_v(x)
=\sum_{v \in V_n} f_{n}(v) \, u_v(x),
\]
where $u_v$ is the unique piecewise affine function on the cable system $\bar{V}_n$ that takes the value 1 on $v$ and zero on the other vertices. We have $0\le u_v\le 1$, $\supp u_v \subset K_n^*(v)$, where  $K_n^*(v)$ is the union of $n$-simplices containing $v$ and
\[
\sum_{v \in V_n} u_v(x)=1, \quad \forall \, x \in K.
\]

Set $g=f-\Phi_n$ and $h=\Phi_n$ so that $f=g+h$.  We claim that $g\in L^p(K,\mu)$ and $h\in W^{1,p}(K)$. Moreover, we claim that both $\|g \|_{L^p(K,\mu)}$ and $\| h \|_{W^{1,p}(K)}$ can be bounded in terms of $E_p(f,r)^{1/p}$ where $r$ has order $3^{-n}$.

We begin with estimating $\|g \|_{L^p(K,\mu)}$. Note that the covering $\{K_n^*(v)\}_{v\in V_n}$ has the bounded overlap property. Also, for any $x\in K_n^*(v)$, there exists a constant $c_1>1$ ($c_1=3$ will do) such that $K_{n+1}^*(v)\subset B(x,c_1 3^{-n})$. Therefore by H\"older's inequality one has
\begin{equation}\label{eq:g-Lp}
\begin{split}
   \|g \|_{L^p(K,\mu)}^p 
   &\le C\sum_{v\in V_n} \int_{K_n^*(v)}|f(x)-f_n(v)|^p (u_v(x))^p d\mu(x)
   \\ &\le C\sum_{v\in V_n} \int_{K_n^*(v)}\fint_{K_{n+1}^*(v)} |f(x)-f(y)|^p d\mu(y) d\mu(x)
   \\ &\le C\int_{K}\fint_{B(x,c_1 3^{-n})} |f(x)-f(y)|^p d\mu(y) d\mu(x).
\end{split}
\end{equation}
%where we denote $\fint_A fd\mu=\frac1{\mu(A)}\int_A fd\mu$.

It remains to control $\| h \|_{W^{1,p}(K)}$. By Proposition  \ref{prop:comparison2}, it is equivalent to estimate the $p$-energy $\mathcal E_p(h)$. 
Since $h$ is an $n$-piecewise affine function, one has $\mathcal E_p^m(h)=\mathcal E_p^n(h)$ for any $m\ge n$ (see Lemma \ref{lem:p-energyPA}). We thus need to estimate 
$\mathcal E_p^n(h)$. Observe that for any $x\in V_n$, one has $h(x)=f_n(x)$ by definition. Hence 
\begin{align*}
    \mathcal E_p^n(h)=\frac123^{(p-1)n}\sum_{x,y\in V_n,x\sim y}|f_n(x)-f_n(y)|^p.
\end{align*}
For any neighboring vertices $x,y \in V_n$, H\"older's inequality yields
\begin{align*}
    |f_n(x)-f_n(y)|
    & \le \frac{1}{\mu(K_{n+1}^*(x))\mu(K_{n+1}^*(y))}\int_{K_{n+1}^*(x)}\int_{K_{n+1}^*(y)}|f(z)-f(w)|d\mu(z)d\mu(w)
    \\ &\le C\left(\frac{1}{5^{2n}}\int_{K_{n+1}^*(x)}\int_{K_{n+1}^*(y)}|f(z)-f(w)|^pd\mu(z)d\mu(w)\right)^{\frac1p}.
\end{align*}
Thanks to the fact that $x,y\in V_n$ are adjacent, there exists a constant $c_2>1$ ($c_2=3$ will do) such that $K_{n+1}^*(y) \subset B(z, c_23^{-n})$ for any $z\in K_{n+1}^*(x)$.

Therefore we get
\[
  |f_n(x)-f_n(y)|^p\le \frac{C}{5^{2n}}\int_{K_{n+1}^*(x)}\int_{B(z, c_2 3^{-n})}|f(z)-f(w)|^pd\mu(z)d\mu(w).
\]
By the bounded overlap property of $\{K_{n+1}^*(v)\}_{v\in V_n}$, we then have
\begin{align*}
    \mathcal E_p^n(h)
    & \le C\frac{3^{(p-1)n}}{5^{2n}}\sum_{x,y\in V_n, x\sim y} \int_{K_{n+1}^*(x)}\int_{K_{n+1}^*(y)}|f(z)-f(w)|^pd\mu(z)d\mu(w)
    \\ &\le  C\frac{3^{(p-1)n}}{5^{2n}}\int_{K}\int_{B(z, c_2 3^{-n})}|f(z)-f(w)|^pd\mu(z)d\mu(w).
\end{align*}
Set $r_n=c_3 3^{-n}$ where $c_3=\max \{ c_1,c_2 \}$. We can rewrite the above inequality as
\[
\mathcal E_p^n(h) \le  \frac{C}{r_n^{p\alpha_p}}\int_{K}\fint_{B(z, r_n)}|f(z)-f(w)|^pd\mu(z)d\mu(w).
\]
Consequently, 
\[
\|h\|_{W^{1,p}(K)}^p \le C\mathcal E_p(h) \le  \frac{C}{r_n^{p\alpha_p}}\int_{K}\fint_{B(z, r_n)}|f(z)-f(w)|^pd\mu(z)d\mu(w).
\]

On the other hand, \eqref{eq:g-Lp} also gives that 
\[
\|g\|_{L^p(K)}\le C\left(\int_K\fint_{B(x,r_n)}|f(x)-f(y)|^pd\mu(y)d\mu(x)\right)^{\frac1p}.
\]
We conclude that for every $t >0$ and $n \ge 1$
\[
K(f,t) \le C \left( 1+\frac{t}{r_n^{\alpha_p}} \right)E_p(f,r_n)^{\frac1p}.
\]
On the other hand the decomposition $f=g+h$ with $h=\int_K f$ yields that for every $t >0$
\[
K(f,t) \le C E_p(f,2)^{1/p}.
\]
The conclusion follows.
\end{proof}

We thus get as a corollary, the theorem stated at the beginning of the section.

\begin{cor}
For every $p > 1$ and $0\le \alpha \le\alpha_p $, we have
\[
\mathcal{B}^{\alpha}_{p,\infty}(K)=(L^p(K,\mu), W^{1,p}(K))_{\alpha/\alpha_p,\infty}.
\]
\end{cor}
\begin{proof}
By Corollary \ref{cor:sup-limsup}
\begin{align*} 
\sup_{r>0}\frac1{r^{\alpha}}\Bigg( \int_K\int_{B(x,r) \cap K}\frac{|f(y)-f(x)|^p}{ \mu(B(x,r))}\, d\mu(y)\, d\mu(x)\Bigg)^{\frac1p} \simeq \sup_{r>0} r^{-\alpha} K(f,r^{\alpha_p})
\simeq \sup_{t>0} t^{-\alpha/\alpha_p} K(f,t).
\end{align*}
\end{proof}

By the reiteration Theorem \ref{reiteration}, we obtain therefore as a corollary the following interpolation result for the Besov-Lipschitz spaces: For $p >1$, $0 \le \theta_1 < \theta_2 \le \alpha_p$, $0<\beta <1$
\begin{align}\label{interpolationB}
\mathcal{B}^{\theta_3}_{p,\infty}(K)=(\mathcal{B}^{\theta_1}_{p,\infty}(K), \mathcal{B}^{\theta_2}_{p,\infty}(K))_{\beta,\infty}, \quad \theta_3=(1-\beta)\theta_2+ \beta \theta_1.
\end{align}
Such interpolation results for the Besov-Lipschitz spaces are not new: We refer to \cite{GKS}, \cite{HanMullerYang}, \cite{MullerYang} and \cite{Yang} for versions of the  interpolation \eqref{interpolationB} in different settings. 

\subsection{Interpolation of the Besov-Lipschitz spaces, $p=1$}

For $p=1$, the endpoint of the interpolation scale is not $W^{1,1}(K)$ but the larger space $BV(K)$ of bounded variation functions that was introduced in \cite{BV3}.

\begin{defn}\label{def:BV}
The Korevaar-Schoen BV space $BV(K)$ is defined by
\[
BV(K)=\left\{ f \in L^1(K,\mu), \,  \limsup_{r\to 0^+} \frac{1}{r^{d_h}}\int_K\int_{B(x,r)}\frac{|f(y)-f(x)|}{ \mu(B(x,r))}\, d\mu(y)\, d\mu(x)<+\infty \right\},
\]
and for $f\in BV(K)$ we define
\[
\|f\|_{BV(K)}:=\limsup_{r\to 0^+} \frac{1}{r^{d_h}}\int_K\int_{B(x,r)}\frac{|f(y)-f(x)|}{ \mu(B(x,r))}\, d\mu(y)\, d\mu(x).\
\]
\end{defn}

\begin{remark}
From Proposition \ref{prop:comparison2} it is clear that $W^{1,1}(K) \subset \mathcal{F}_1 \subset BV(K)$. However, all the inclusions are strict since $BV(K)$ also contains non-continuous functions, see \cite{BV3}.
\end{remark}

\begin{thm}
 For  $0\le \alpha \le d_h $,
\[
\mathcal{B}^{\alpha}_{1,\infty}(K)=(L^1(K,\mu), BV(K))_{\alpha/d_h,\infty}.
\]
\end{thm}

\begin{proof}
The proof is relatively similar to that of Theorem \ref{interpolation} so we will omit the details but focus on the main ingredients. The first ingredient which is proved in \cite{BV3} for any nested fractal using heat kernel methods is the estimate
\begin{align*}
  \sup_{r >0} \frac{1}{r^{d_h}} \int_K\int_{B(x,r) }\frac{|f(y)-f(x)|}{ \mu(B(x,r))}\, d\mu(y)\, d\mu(x)   \le  C \limsup_{r\to 0^+} \frac{1}{r^{d_h}} \int_K\int_{B(x,r) }\frac{|f(y)-f(x)|}{ \mu(B(x,r))}\, d\mu(y)\, d\mu(x).
\end{align*}
The second ingredient is Proposition  \ref{prop:comparison2} for $p=1$ and when $f$ is a piecewise affine function.
\end{proof}

\subsection{Real interpolation  of the Sobolev spaces}\label{inter sobo}

The interpolation with respect to the parameter $p$ is easier in view of the characterization of $W^{1,p}(K)$ given in Theorem \ref{carac Sobolev}.

\begin{thm}\label{real inter sobo}
For $1 \le p_1 <p < p_2 \le +\infty$,
\[
W^{1,p}(K)=(W^{1,p_1}(K),W^{1,p_2}(K))_{\theta,p}
\]
where $\theta \in (0,1)$ is such that
\[
\frac{1}{p}=\frac{1-\theta}{p_1}+\frac{\theta}{p_2}.
\]
\end{thm}
\begin{proof}
For every $1 \le p \le +\infty$ the map $ f \to \partial f$ is a bi-Lipschitz isomorphism $W_0^{1,p}(K) \to L^p(K,\nu)$, where $W_0^{1,p}(K)=\{ f \in W^{1,p}(K), f(0)=0 \}$. The measure $\nu$ is sigma-finite, and therefore
\[
L^{p}(K,\nu)=(L^{p_1}(K,\nu),L^{p_2}(K,\nu))_{\theta,p}.
\]
The result follows.
\end{proof}

\bibliographystyle{plain}
\bibliography{mybib.bib}
\

\noindent
Fabrice Baudoin: \url{fabrice.baudoin@uconn.edu}\\
Department of Mathematics,
University of Connecticut,
Storrs, CT 06269

\

\noindent Li Chen: \url{lichen@lsu.edu}\\
Department of Mathematics, Louisiana State University, Baton Rouge, LA 70803
\end{document}